\documentclass[reqno, 12pt]{amsart}
\pdfoutput=1
\makeatletter
\let\origsection=\section \def\section{\@ifstar{\origsection*}{\mysection}} 
\def\mysection{\@startsection{section}{1}\z@{.7\linespacing\@plus\linespacing}{.5\linespacing}{\normalfont\scshape\centering\S}}
\makeatother 

\usepackage{amsmath,amssymb,amsthm}
\usepackage{mathrsfs}
\usepackage{mathabx}\changenotsign
\usepackage{dsfont}
 
\usepackage{xcolor}
\usepackage[backref]{hyperref}
\hypersetup{
    colorlinks,
    linkcolor={red!60!black},
    citecolor={green!60!black},
    urlcolor={blue!60!black}
}

\usepackage{graphicx}

\usepackage[open,openlevel=2,atend]{bookmark}

\usepackage[abbrev,msc-links,backrefs]{amsrefs} 
\usepackage{doi}

\renewcommand{\PrintDOI}[1]{\doi{#1}}

\usepackage[T1]{fontenc}
\usepackage{lmodern}
\usepackage[babel]{microtype}
\usepackage[english]{babel}

\usepackage{geometry}
\numberwithin{equation}{section}
\numberwithin{figure}{section}

\usepackage{enumitem}
\def\rmlabel{\upshape({\itshape \roman*\,})}

\let\polishlcross=\l
\def\l{\ifmmode\ell\else\polishlcross\fi}

\let\emptyset=\varnothing
\let\setminus=\smallsetminus

\makeatletter
\def\moverlay{\mathpalette\mov@rlay}
\def\mov@rlay#1#2{\leavevmode\vtop{   \baselineskip\z@skip \lineskiplimit-\maxdimen
   \ialign{\hfil$\m@th#1##$\hfil\cr#2\crcr}}}
\newcommand{\charfusion}[3][\mathord]{
    #1{\ifx#1\mathop\vphantom{#2}\fi
        \mathpalette\mov@rlay{#2\cr#3}
      }
    \ifx#1\mathop\expandafter\displaylimits\fi}
\makeatother

\DeclareFontFamily{U}  {MnSymbolC}{}
\DeclareSymbolFont{MnSyC}         {U}  {MnSymbolC}{m}{n}
\DeclareFontShape{U}{MnSymbolC}{m}{n}{
    <-6>  MnSymbolC5
   <6-7>  MnSymbolC6
   <7-8>  MnSymbolC7
   <8-9>  MnSymbolC8
   <9-10> MnSymbolC9
  <10-12> MnSymbolC10
  <12->   MnSymbolC12}{}
\DeclareMathSymbol{\powerset}{\mathord}{MnSyC}{180}

\theoremstyle{plain}
\newtheorem{theorem}{Theorem}[section]

\newtheorem{lemma}[theorem]{Lemma}

\theoremstyle{definition}

\newtheorem{claim}[theorem]{Claim}

\theoremstyle{remark}

\newcommand{\ang}[1]{\left\langle #1\right\rangle}

\newcommand{\G}{\mathcal{G}}

\newcommand{\Hom}{\mathrm{Hom}}
\newcommand{\EE}{\mathds{E}}

\newcommand{\RR}{\mathds{R}}

\newcommand{\HH}{\mathcal{H}}

\newcommand{\C}{\mathcal{C}}

\newcommand{\E}{\mathcal{E}}

\newcommand{\FF}{\mathcal{F}}

\begin{document}

\title{Odd cycles in subgraphs of sparse pseudorandom graphs}

\author{S\"oren Berger}
\author{Joonkyung Lee}
\address{Fachbereich Mathematik, Universit\"at Hamburg, Hamburg, Germany}
\email{\{soeren.berger,\,joonkyung.lee,\,mathias.schacht\}@uni-hamburg.de}

\author{Mathias Schacht}
\address{Department of Mathematics, Yale University, New Haven, USA}
\email{mathias.schacht@yale.edu}
\thanks{Research of the first two authors was supported by ERC Consolidator Grant PEPCo 724903.}

\date{}


\begin{abstract}
We answer two extremal questions about odd cycles that naturally arise in the study of sparse pseudorandom graphs.
Let $\Gamma$ be an $(n,d,\lambda)$-graph, i.e., $n$-vertex, $d$-regular graphs with all nontrivial eigenvalues in the interval $[-\lambda,\lambda]$.
Krivelevich, Lee, and Sudakov conjectured that, whenever $\lambda^{2k-1}\ll d^{2k}/n$, every subgraph $G$ of $\Gamma$ with $(1/2+o(1))e(\Gamma)$ edges contains an odd cycle $C_{2k+1}$.
Aigner-Horev, H\`{a}n, and the third author proved a weaker statement by allowing an extra polylogarithmic factor in the assumption $\lambda^{2k-1}\ll d^{2k}/n$,
but we completely remove it and hence settle the conjecture.
This also generalises Sudakov, Szabo, and Vu's Tur\'{a}n-type theorem for triangles. 

Secondly, we obtain a Ramsey multiplicity result for odd cycles. Namely, in the same range of parameters, we prove that every 2-edge-colouring of $\Gamma$ contains at least $(1-o(1))2^{-2k}d^{2k+1}$ monochromatic copies of $C_{2k+1}$.
Both results are asymptotically best possible by Alon and Kahale's construction of $C_{2k+1}$-free pseudorandom graphs.
\end{abstract}

\maketitle

\section{Introduction}
In the last two decades, one of the major developments in extremal and probabilistic combinatorics has been the study of sparse (pseudo)random analogue of classical results. 
We continue to study analogues of classical theorems in sparse pseudorandom graphs.
An \emph{$(n,d,\lambda)$-graph} $\Gamma$ is a $d$-regular $n$-vertex graph such that the spectrum $d = \lambda_1 \ge \dots \ge \lambda_n$ of its adjacency matrix~$A_\Gamma$ satisfies $|\lambda_i|\leq\lambda$ for $i=2,3,\dots,n$.
Although this is one of the most well-known examples of pseurandom graphs and hence received considerable attention, as surveyed in~\cite{KS},
there are only very few analogues of classical theorems for $(n,d,\lambda)$-graphs.
For example, Sudakov, Szabo, and Vu~\cite{SSzV} proved an analogue of Tur\'an's theorem for $(n,d,\lambda)$-graphs,
where the range of parameters is believed to be optimal (for  other extremal or Ramsey-type results in this context see, e.g.,~\cites{ABSS,CFZ} and references therein).

We prove two analogues of classical results for $(n,d,\lambda)$-graphs that concern odd cycles~$C_{2k+1}$.
The range of parameters we focus on is always $\lambda^{2k-1} \ll {d^{2k}}/{n}$, which is tight for each $C_{2k+1}$ in the sense that there exists a $C_{2k+1}$-free $(n,d,\lambda)$-graph with $\lambda^{2k-1}=\Theta(d^{2k}/n)$ by the construction by Alon and Kahale~\cite{AK}, built on Alon's triangle-free pseudorandom graphs~\cite{A94}.

We study the Ramsey multiplicity of odd cycles in $(n,d,\lambda)$-graphs.
Let $N_{H}(G)$ be the number of labelled copies of $H$ in $G$.
A graph $H$ is \emph{common} if the number of monochromatic $H$-copies in a 2-edge-colouring of $K_n$ is minimised by the random colouring, i.e.,
\[
N_H(G) +N_H(\overline{G})\geq (1-o(1))2^{1-e(H)}n^{|V(H)|},
\]
whenever $G$ is an $n$-vertex graph and $\overline{G}$ is its complement.
In 1962, Erd\H{o}s~\cite{E62common} conjectured that every complete graph is common,
which is generalised by Burr and Rosta~\cite{BR80} for arbitrary graphs instead of complete graphs.
However, already the original Erd\H{o}s conjecture turned out to be false, as was shown by Thomason~\cite{Thom89} for every $K_t$, $t\geq 4$.
There are many common and uncommon graphs known since then~\cites{HHKNR12,JST96,Sid96}, although the complete classification is far beyond our reach.
In particular, Sidorenko~\cite{Sid89} proved that every odd cycle is common.
We obtain a sparse pseudorandom analogue of Sidorenko's theorem.
\begin{theorem}\label{thm:main_common}
Let $\varepsilon>0$ and let $\Gamma$ be an $(n,d,\lambda)$-graph.
Then there exists $\eta>0$ such that, whenever $\lambda^{2k-1} \le \eta {d^{2k}}/{n}$ and $G$ is a subgraph of $\Gamma$,
\[
N_{C_{2k+1}}(G) +N_{C_{2k+1}}(\Gamma\setminus G)\geq (1-\varepsilon)2^{-2k}d^{2k+1}.
\]
\end{theorem}

Secondly, we prove an analogue of the Erd\H{o}s--Stone theorem for odd cycles,
stating that every $n$-vertex graph with more than half of the all possible edges must contain a copy of an odd cycle of fixed length.
Theorem~\ref{thm:main_turan} below yields the same conclusion for subgraphs of suitable $(n,d,\lambda)$-graphs with relative density $1/2+o(1)$.
Obviously, Alon and Kahale's $C_{2k+1}$-free graphs do not possess the Erd\H{o}s--Stone property for $C_{2k+1}$ and Krivlevich, Lee, and Sudakov~\cite{KLS10} conjectured that the example by Alon and Kahale is asymptotically optimal. 
We verify this conjecture.
\begin{theorem}\label{thm:main_turan}
Let $k\geq 1$ be an integer and let $\delta> 0$.
Then there exist $\eta>0$ and $n_0$ such that the following holds:
let $n  \ge n_0$ and let $\Gamma$ be an $(n,d,\lambda)$-graph satisfying $\lambda^{2k-1} \le \eta {d^{2k}}/{n}$.
If $G\subset \Gamma$ is a subgraph such that $e(G) \ge  \left( \tfrac{1}{2} + \delta \right) \tfrac{d}{n} \binom{n}{2}$, then there is a copy of $C_{2k+1}$ in~$G$.
\end{theorem}
A similar result with a slightly stronger condition $\lambda^{2k-1}(\log n)^{(2k-1)(2k-2)}\ll d^{2k}/n$ was obtained by
Aigner-Horev, H\`{a}n, and the third author~\cite{AHS14}. However, those authors obtained such a result in the more general context of \emph{bijumbled} graphs, while we make use of the spectral estimate for the number of even cycles in $(n,d,\lambda)$-graphs (see Lemma~\ref{lem: spectral argument} below).

Our proof of Theorem~\ref{thm:main_turan} uses a stronger variant (see Theorem~\ref{thm:common_injective_count}) of the first main result, Theorem~\ref{thm:main_common}.
This is a new approach for the Erd\H{o}s-Stone-type problems in pseudorandom setting. However, one cannot expect an analogous solution to the variant of Tur\'{a}n's theorem proved by Sudakov, Szabo, and Vu~\cite{SSzV}, since Thomason~\cite{Thom89} showed that any $K_t$, $t\geq 4$, is uncommon.

\section{Preliminaries}
Throughout this paper, $\Gamma$ always denotes the $(n,d,\lambda)$ graph and $1_\Gamma(x,y)$ is the indicator function of the edge set $E(\Gamma)$. For brevity, $p=d/n$ denotes the edge density of $\Gamma$. We use the standard notation $f(n)\ll g(n)$ if $f(n)/g(n)\to 0$ as $n\to\infty$. We will also write $x=a\pm b$ if and only if $a-b\leq x\leq a+b$.
For each $k> 2$, $C_{k}$ denotes the cycle of length~$k$ and $C_2$ means the single edge graph $K_2$. We denote by $P_k$ the $k$-edge path on~$k+1$ vertices.

In what follows, we shall use the fact $1\ll d$ and $\lambda\ll d$, which are trivial consequences of the crucial condition $\lambda^{2k-1}\ll d^{2k}/n$.
The number of vertices $n=|V(\Gamma)|$ will be taken large enough.

When counting $H$-copies in $G$, it is often convenient to allow possibly degenerate copies of~$H$. 
For graphs $H$ and~$G$, denote by $h_H(G)$ the number of all homomorphisms from~$H$ to~$G$.
Let the \emph{graph homomorphism density} $t_H(G):=h_H(G)/|V(G)|^{|V(H)|}$,
that is, the number of homomorphisms from $H$ to~$G$ divided by the number of vertex maps from~$H$ to~$G$.
Indeed, the graph homomorphism density defined above naturally generalises to (not necessarily nonnegative) weighted graphs, i.e., for a symmetric function $f\colon V(G)^2\rightarrow \RR$,
\begin{align*}
    t_H(f):=\EE \bigg[ \prod_{ij \in E(H)} f(x_i,x_j) \bigg],
\end{align*}
where each $x_i$ is a uniform random vertex in $V(G)$ chosen independently.
We shall repeatedly use a key pseudorandom property of an $(n,d,\lambda)$-graph, given by the Expander Mixing Lemma.
\begin{lemma}[Expander Mixing Lemma]\label{lem:EML}
Let $\Gamma$ be an $n$-vertex graph whose nontrivial eigenvalues lie in the interval $[-\lambda,\lambda]$. Then for every weight function $u,v\colon V(\Gamma)\rightarrow [0,1]$, 
\begin{align}\label{iq:weighted-bijumbled}
	\left| \sum_{x,y \in V(\Gamma)}u(x)1_\Gamma(x,y)v(y) - \frac{d}{n} \sum_{x\in V(\Gamma)} u(x)  \sum_{y\in V(\Gamma)} v(y) \right| \le \lambda\sqrt{\sum_{x\in V(\Gamma)} {u(x)}^2 \sum_{y\in V(\Gamma)} {v(y)}^2}.
\end{align}
\end{lemma}
When $u$ and $v$ are $\{0,1\}$-valued, it appeared in~\cites{AM,T}. Our weighted version of the lemma can easily be derived by following the standard proofs of theirs.

The Expander Mixing Lemma yields an estimate on $h_{C_{2k+1}}(\Gamma)$ for every fixed $k$. For that fix a vertex, say $1$ in $V(C_k)=[k]$, and let $h_{C_{k}}(\Gamma;x)$ be the number of homomorphic copies of~$C_{k}$ that maps $1$ to $x\in V(\Gamma)$. Let $w_{k,\Gamma}(x,y)$ be the number of $k$-edge walks from~$x$ to $y$ in $\Gamma$. Then
\[
h_{C_{2k+1}}(\Gamma;x)=\sum_{y,z\in V(\Gamma)}w_{k,\Gamma}(x,y)1_\Gamma(y,z)w_{k,\Gamma}(x,z).
\]
Since $\sum_{y\in V(\Gamma)}w_{k,\Gamma}(x,y)=d^k$, the Expander Mixing Lemma yields
\[
\left| h_{C_{2k+1}}(\Gamma;x) - \frac{d^{2k+1}}{n} 
\right| 
\le \lambda \sum_{y\in V(\Gamma)} {w_{k,\Gamma}(x,y)}^2=\lambda \cdot h_{C_{2k}}(\Gamma;x).
\]
Summing over all $x\in V(\Gamma)$ hence gives 
\begin{align}\label{eq:oddcycle_estimate}
    \left| h_{C_{2k+1}}(\Gamma) - d^{2k+1}
\right| 
\le \lambda \cdot h_{C_{2k}}(\Gamma).
\end{align}

In the following section, we shall prove a slightly stronger statement, Theorem~\ref{thm:common_injective_count}, than Theorem~\ref{thm:main_common} by considering an `almost-regular' subgraph of $\Gamma$ induced on a large vertex subset instead of the $d$-regular graph $\Gamma$.
To this end, we say that a vertex subset $X\subseteq V(\Gamma)$ is \emph{$\delta$-almost-regular} if
\[
\deg_{\Gamma[X]}(x)= (1 \pm \delta) p |X|\text{ for all }x \in X.
\]
In particular, $V(\Gamma)$ is $\delta$-almost-regular for any $\delta>0$.
Indeed, we may replace $\Gamma$ by $\Gamma[X]$ in proving~\eqref{eq:oddcycle_estimate} to obtain an analogous bound. As
\[
h_{C_{2k+1}}(\Gamma[X];x)=\sum_{y,z\in X}w_{k,\Gamma[X]}(x,y)1_\Gamma(y,z)w_{k,\Gamma[X]}(x,z),
\]
the Expander Mixing Lemma gives
\[
\left| h_{C_{2k+1}}(\Gamma[X];x) - p\cdot d_{k}(x;\Gamma[X])^2 
\right| 
\le \lambda \cdot h_{C_{2k}}(\Gamma[X];x),
\]
where $d_{k}(x;\Gamma[X])$ denotes the number of $k$-edge walks in $\Gamma[X]$ starting at $x\in X$. Since~$X$ is $\delta$-almost-regular, $d_{k}(x;\Gamma[X])=(1\pm k\delta)p^{k}|X|^{k}$ for every $x\in X$.
Thus, we obtain
\begin{align}\label{eq:oddcycle_X}
    \left| h_{C_{2k+1}}(\Gamma[X]) - p^{2k+1}|X|^{2k+1}\right| 
\le \lambda \cdot h_{C_{2k}}(\Gamma) +2k\delta p^{2k+1}|X|^{2k+1}.
\end{align}

To bound the right-hand side above, we shall use the following spectral argument.

\begin{lemma}\label{lem: spectral argument}
Let $\Gamma$ be an $(n,d,\lambda)$-graph and let $k$ be a positive integer. Then
\begin{align*}
	h_{C_{2k}}(\Gamma)	\le d^{2k} + \lambda^{2k-2}dn.
\end{align*}
\end{lemma}
\begin{proof}
Since $h_{C_{k}}(\Gamma)=\mathrm{tr}(A^k_\Gamma)=\lambda_1^{k}+\lambda_2^{k}+\dots+\lambda_n^{k}$ for every $k\geq 2$,
\[
h_{C_{2k}}(\Gamma)=\lambda_1^{2k}+\lambda_2^{2k}+\dots+\lambda_n^{2k}
\leq d^{2k} +\lambda^{2k-2}(\lambda_2^2+\dots+\lambda_n^2) \leq d^{2k}+\lambda^{2k-2}dn,
\]
where the last inequality is from $\lambda_1^{2}+\lambda_2^{2}+\dots+\lambda_n^{2}=h_{K_2}(\Gamma)=dn$.
\end{proof}

Note that the assumption $\lambda^{2k-1}\ll d^{2k}/n$ in Theorem~\ref{thm:main_common} and~\ref{thm:main_turan} combined with the fact $\lambda\ll d$ implies $d^{2k+1}\gg \lambda\cdot h_{C_{2k}}(\Gamma)$ and hence, \eqref{eq:oddcycle_estimate} implies $h_{C_{2k+1}}(\Gamma)=(1\pm o(1))d^{2k+1}$.
Similarly, if $|X|\geq \mu n$, \eqref{eq:oddcycle_X} gives
\begin{align*}
    \left| h_{C_{2k+1}}(\Gamma[X]) - p^{2k+1}|X|^{2k+1}\right| 
\le \lambda p^{2k}n^{2k}+\lambda^{2k-1}pn^2 +k\delta p^{2k+1}|X|^{2k+1}.
\end{align*}
In particular,
\begin{align}\label{eq:oddcycle_final}
     \nonumber h_{C_{2k+1}}(\Gamma[X]) &\ge  p^{2k+1}|X|^{2k+1}\left(1- \frac{\lambda n^{2k} }{p|X|^{2k+1}}-\frac{\lambda^{2k-1}n^2}{p^{2k}|X|^{2k+1}} -2k\delta\right)\\
     &\ge  p^{2k+1}|X|^{2k+1}\left(1- \frac{\lambda  }{\mu^{2k+1}d}-\frac{\lambda^{2k-1}n}{\mu^{2k+1}d^{2k}} -2k\delta\right),
\end{align}
which essentially means $h_{C_{2k+1}}(\Gamma[X])\geq (1-o(1))(p|X|)^{2k+1}$.

The following lemma will be useful in proving that the number of the degenerate copies of an odd cycle $C_{2k+1}$ is negligible.
\begin{lemma}\label{lem:degenerate}
Let $H$ be the graph consisting of edge-disjoint $C_{2q}$ and $C_{2r+1}$ sharing exactly one vertex. Then
\begin{align*}
    h_H(\Gamma)\leq 
\frac{1}{n}d^{2(q+r)+1}+\lambda^{2q-2}d^{2r+2}+\lambda d^{2(q+r)}+\lambda^{2(q+r)-1}dn.
\end{align*}
\end{lemma}
\begin{proof}
For each homomorphism $\phi\in\Hom(C_{2q},\Gamma)$, let $h_H(\Gamma;\phi)$ be the number of homomorphisms from $H$ to $\Gamma$ that extends $\phi$
and let $w_r(x;\phi)$ be the number of $r$-edge walks from the image of the shared vertex $v$ under $\phi$ to $x\in V(\Gamma)$. Then 
\[
h_{H}(\Gamma;\phi)=\sum_{x,y\in V(\Gamma)}w_r(x;\phi)\gamma(x,y)w_r(y;\phi).
\]
The Expander Mixing Lemma gives
\[
\left| h_{H}(\Gamma;\phi) - \frac{d^{2r+1}}{n} 
\right| 
\le \lambda \sum_{x\in V(\Gamma)} {w_{r}(x;\phi)}^2.
\]
Note that $\sum_{x\in V(\Gamma)}w_{r}(x;\phi)^2$ counts the number of homomorphisms from another graph~$H'$ obtained by $C_{2q}$ and $C_{2r}$ identified on the vertex $v$ that extends $\phi$.
In particular, this is a degenerate copy of $C_{2(q+r)}$. Thus, summing above over all $\phi\in\Hom(C_{2q},H)$ yields
\[
\left|h_H(\Gamma)-\frac{d^{2r+1}}{n}h_{C_{2q}}(\Gamma)\right|\leq \lambda\cdot h_{C_{2(q+r)}}(\Gamma)
\]
and applying Lemma~\ref{lem: spectral argument} concludes the proof.
\end{proof}
If $q+r=k$ and $\lambda^{2k-1}\ll d^{2k}/n$, then $h_H(\Gamma)\ll d^{2k+1}$. Whenever a homomorphic copy of~$C_{2k+1}$ is degenerate, it induces a homomorphic copy $H$ of two shorter cycles sharing one vertex. Hence, Lemma~\ref{lem: spectral argument} shows that most of the homormophic copies of $C_{2k+1}$ are nondegenerate.

\section{The relative commonality of odd cycles}

We shall prove the following slightly stronger statement than Theorem~\ref{thm:main_common}.
To avoid ambiguity in the normalising factor, $t_H(G)$ means $h_H(G)/|X|^{|V(H)|}$ whenever $G$ is a subgraph of~$\Gamma[X]$. 
\begin{theorem}\label{thm:common_injective_count}
For $0<\mu,\delta<1$ and an integer $k\geq1$, there exists $\eta = \eta(\delta,\mu,k)>0$ such that 
the following holds: let $\Gamma$ be an $(n,d,\lambda)$-graph satisfying $\lambda^{2k-1} \le \eta {d^{2k}}/{n}$ and let $X$ be a $\delta$-almost-regular vertex subset of $\Gamma$ 
with $|X| \ge \mu n$.
Then for every subgraph $G$ of $\Gamma[X]$,  we have
$$
N_{C_{2k+1}}(G)+N_{C_{2k+1}}(\Gamma[X]\setminus G)\geq \frac{1}{2^{2k}} (p|X|)^{2k+1}\left( 1 - 2^{8k}\delta \right).$$ 
\end{theorem} 
We remark that this strenghtening of Theorem~\ref{thm:main_common} is purely for the future purpose to derive Theorem~\ref{thm:main_turan}
and we did not attempt to optimise the constants.
The key ingredient in proving Theorem~\ref{thm:common_injective_count} is a homomorphism counting lemma. 

\begin{lemma}\label{lem:homcounting}
Let $\delta,\mu>0$ and let $\Gamma$ be an $(n,d,\lambda)$-graph.
For every $\delta$-almost-regular subset $X\subset V(\Gamma)$ with $|X|\geq \mu n$
and every subgraph $G$ of $\Gamma[X]$,
\begin{align*}
	t_{C_{2k+1}}(G)+t_{C_{2k+1}}(\Gamma[X] \setminus G)
	\ge\frac{1}{2^{2k}} p^{2k+1} \left( 1 - 2^{7k} \left( \delta +  \frac{\lambda}{\mu^{2k} d}  + \frac{\lambda^{2k-1}n}{\mu^{2k}d^{2k}} \right) \right).
\end{align*}
\end{lemma}

Theorem~\ref{thm:common_injective_count} can easily be deduced by setting $\eta=\mu^{4k^2}\delta^{2k}/10^{2k}$ in the lemma above and the fact that there are at most
$2(2k+1)^2 \eta d^{2k+1}$
degenerate copies of $C_{2k+1}$ by Lemma~\ref{lem:degenerate}.

Throughout this section, we write $\gamma_X:=1_{E(\Gamma[X])}$ or even $\gamma=\gamma_X$ if $X$ is clear from the context.
Similary, let $g=g_X$ be the indicator of the edges in the subgraph $G$ of $\Gamma[X]$.

Let $J$ be an edge subset of $H$. For $f_1,f_2 \colon X^2 \to \RR$ and $x_1,\dots,x_{|V(H)|} \in X$, write
\begin{align*}
	\langle f_1,f_2 \rangle^J_H
	:= \prod_{ij \in J} f_1(x_i,x_j) \prod _{ij \in E(H) \setminus J} f_2(x_i,x_j).
\end{align*}
In fact, $t_H(f)=\EE \big[ \langle f,h \rangle^{E(H)}_H \big]$ for any $h$.
For $\alpha,\beta\in\RR$ we may expand $t_H(\alpha f_1+\beta f_2)$ to
\[
 t_H(\alpha f_1+\beta f_2) = \sum_{J\subseteq E(H)}\alpha^{|J|}\beta^{e(H)-|J|} \EE \left[\ang{f_1,f_2}_{H}^{J}\right].
\]
For brevity, write
\begin{align*}
	\mathcal{E} (H) :=  \left\lbrace J \subset E(H) \colon|J| \hbox{ is even} \right\rbrace
	\ \ \text{and}\ \ 	
	\mathcal{E}_+ (H) :=  \left\lbrace J \subset E(H)\colon |J| \hbox{ is even and nonzero} \right\rbrace.
\end{align*}
Let $f:=2g-\gamma$ so that $g=\tfrac{1}{2}(f+\gamma)$ and $\gamma-g=\tfrac{1}{2}(-f+\gamma)$. Since $0\leq g\leq \gamma$ we have $|f|\leq \gamma$. Moreover, from the definition of $f$ it follows that
\begin{equation}\label{eq:cancel}
	t_H(g)+t_H(\gamma-g)
= \bigg( \frac{1}{2} \bigg)^{e(H)-1} \bigg( t_H(\gamma)+\!\!\sum_{J \in \mathcal{E}_+(H)} \EE \big[ \ang{f, \gamma}_H^J \big] \bigg).
\end{equation}
Recall that \eqref{eq:oddcycle_final} implies
\begin{align}\label{eq:homdensity}
    t_{C_{2k+1}}(\gamma)\ge  p^{2k+1}\left(1- \frac{\lambda  }{\mu^{2k+1}d}-\frac{\lambda^{2k-1}n}{\mu^{2k+1}d^{2k}} -2k\delta\right).
\end{align}
Thus, in order to prove Lemma~\ref{lem:homcounting}, it suffices to show that $\EE \big[ \ang{f, \gamma}_{C_{2k+1}}^J \big]$ is `almost nonnegative'.
For that we generalise Sidorenko's arguments~\cite{Sid89} for proving the commonality of odd cycles.
For a symmetric function $f\colon X^2\to\RR$,
define a polynomial in $\RR[z]$
\[
Q_{H}(z;f):=\sum_{J\in\mathcal{E}_{+}(H)} \EE\big[\ang{f,z\gamma}_{H}^J\big]
=\sum_{J\in\mathcal{E}_{+}(H)} \EE\big[\ang{f,\gamma}_{H}^J\big] z^{e(H)-|J|}.
\]
\begin{lemma}\label{lem:diff}
Suppose $|f(x,y)|\le \gamma(x,y)$ for every $x,y\in X$. Then
\begin{align*}
    \left|\frac{d}{dz}Q_{C_{2k+1}}(z;f)-p(2k+1)Q_{P_{2k}}(z;f)\right|\leq (2k+1)p^{2k+1}\left(\frac{\lambda}{\mu^{2k+1}d}+\frac{\lambda^{2k-1}n}{\mu^{2k+1}d^{2k}}\right).
\end{align*}
\end{lemma}
\begin{proof}
Since $\{\E_+(C_{2k+1}\setminus e)\}_{e\in E(C_{2k+1})}$ covers each $J\in \E_+(C_{2k+1})$ exactly $2k+1-|J|$ times,
\begin{align}\label{eq:diff}
  \nonumber  \frac{d}{dz}Q_{C_{2k+1}}(z;f)&=\sum_{J\in\mathcal{E}_{+}(H)} \EE\left[\ang{f,\gamma}_{C_{2k+1}}^J\right] (2k+1-|J|)z^{2k-|J|}\\
    &=\sum_{e\in E(C_{2k+1})}\sum_{J\in \E_+(C_{2k+1}\setminus e)} \EE\left[\ang{f,\gamma}_{C_{2k+1}}^J\right] z^{2k-|J|}.
\end{align}
As $C_{2k+1}\setminus e$ is always isomorphic to $P_{2k}$, we regard $J$ as a subgraph of $P_{2k}$ on $[2k+1]$ with edges $\{i,i+1\}$, $i=1,2,\dots,2k$.
Let $L$ and $R$ be the edges in $P_{2k}$ induced on vertices $\{1,2,\dots,k+1\}$ and $\{k+1,\dots,2k+1\}$.
For each $z\in X$, let $\ell_z,r_z \colon V(\Gamma) \to \RR$ be
\begin{align*}
	\ell_z(x)& :=\sum_{\substack{x_{k+1}=z, x_{1}=x,\\ x_i\in X, 1<i\leq k}}\prod_{ij\in L\cap J}f(x_i,x_j)\prod_{ij\in L\setminus J}\gamma(x_i,x_j)\\
~\text{ and }~	r_z(x) &:= \sum_{\substack{x_{k+1}=z, x_{2k+1}=x,\\ x_i\in X, k+2\leq i< 2k+1}}\prod_{ij\in R\cap J}f(x_i,x_j)\prod_{ij\in R\setminus J}\gamma(x_i,x_j).
\end{align*}
Now the Expander Mixing Lemma together with the fact $|f|\leq \gamma\leq 1_\Gamma$ gives
\begin{align*}
	\left| \sum_{x,y \in V(\Gamma) }\ell_z(x) 1_\Gamma(x,y) r_z(y) - p \sum_{x\in V(\Gamma)}\ell_z(x)  \sum_{y \in V(\Gamma)} r_z(y) \right|
\leq \lambda \cdot h_{C_{2k}}(\Gamma;z).
\end{align*}
Since $$\sum_{x,y \in V(\Gamma) }\ell_z(x) 1_\Gamma(x,y) r_z(y)=|X|^{2k}\EE\left[\ang{f,\gamma}_{C_{2k+1}}^J\middle|x_{k+1}=z\right]$$
and
$$\sum_{x\in V(\Gamma)}\ell_z(x)  \sum_{y \in V(\Gamma)} r_z(y)=|X|^{2k}\EE\left[\ang{f,\gamma}_{P_{2k}}^J\middle|x_{k+1}=z\right],$$
Lemma~\ref{lem: spectral argument} gives
\begin{align*}
    \left|\EE\left[\ang{f,\gamma}_{C_{2k+1}}^J\right]-p\cdot \EE\left[\ang{f,\gamma}_{P_{2k}}^J\right]\right|
    &\leq \frac{1}{|X|^{2k+1}}\left(\lambda p^{2k}n^{2k}
    +\lambda^{2k-1}pn^2\right)\\
    &\leq p^{2k+1}\left(\frac{\lambda}{\mu^{2k+1}d}+\frac{\lambda^{2k-1}n}{\mu^{2k+1}d^{2k}}\right).
\end{align*}
Substituting this into \eqref{eq:diff} yields the desired bound.
\end{proof}

Lemma~\ref{lem:diff} roughly means $\frac{d}{dz}Q_{C_{2k+1}}(z;f)\approx p(2k+1)Q_{P_{2k}}(z;f)$ and the next lemma proves $Q_{P_{2k}}(z;f)$ is `almost nonnegative', which will immediately prove that $Q_{C_{2k+1}}(1;f)$ is almost nonnegative too, as planned.

\begin{lemma}\label{lem:better_sidorenko_lemma}
Let $0\le z\le 1$ and let $\gamma=1_{E(\Gamma[X])}$ for a $\delta$-almost-regular set $X$. Suppose $f\colon X^2\to [0,1]$ satisfies $|f(x,y)|\leq \gamma(x,y)$ for all $x,y\in X^2$. Then
\begin{align*}
	\EE\bigg[\sum_{J\in\mathcal{E}_+(C_{2k})} \langle f,z\gamma \rangle_{P_{2k}}^J \bigg]\ge - p^{2k} 2^{5k}\delta.
\end{align*}
\end{lemma}

\begin{proof}
We firstly classify the nonempty edge subsets of $E(P_{2k})$ 
in terms of the first and the last edge in $J$. 
Namely, for nonempty $J\subseteq E(P_{2k})$, let $a_J$ be the smallest $i$ such that $\{i,i+1\}\in J$ and let $b_J$ be the largest $j$ such that $\{j,j+1\}\in J$. Define
\begin{align*}
\C_{i,j} := \lbrace J \subseteq E(P_{2k}) \colon J\neq\emptyset, a_J=i,\text{ and } b_J=j \rbrace.
\end{align*}
and let
\begin{align*}
    S_{i,j}
    :=\EE\bigg[\sum_{J \in \mathcal{C}_{i,j} \cap \E_+(P_{2k})}\langle f,z\gamma \rangle_{P_{2k}}^J \bigg].
\end{align*}
We regard $J\in\C_{i,j}$ as a subset of $E(P_{i,j})$, where $P_{i,j}$ is the path on $\{i,\dots,j+1\}$. 
Then
\begin{align*}
S_{i,j}
&=\EE\bigg[\sum_{J \in \mathcal{C}_{i,j} \cap \E_+(P_{2k})}\ang{f,z\gamma}_{P_{2k}}^J \bigg]\\
&=\sum_{J \in \mathcal{C}_{i,j} \cap \E_+(P_{2k})}\EE\left[\EE\left[\ang{f,z\gamma} _{P_{2k}}^J \middle| x_\ell:i\leq \ell\leq j+1\right] \right]\\
&=\sum_{J \in \mathcal{C}_{i,j} \cap \E_+(P_{2k})}
(1\pm 2k\delta)(pz)^{2k+i-j-1}\EE\left[\ang{f,z\gamma}_{P_{i,j}}^J\right].
\end{align*}
For $m=j-i+1$, let $$T_m:=\sum_{J \in \mathcal{C}_{i,j} \cap \E_+(P_{2k})}
(pz)^{2k+i-j-1}\EE\left[\ang{f,z\gamma}_{P_{i,j}}^J\right].$$
This is well-defined because the right-hand side above only depends on $j-i$.
Let $T_0=T_1=0$ for notational convenience.
Since $|f|\leq \gamma$ and $0\leq z \leq 1$,
\begin{align*}
\left|S_{i,j}
-T_{j-i+1}\right|
&\leq 2k\delta(pz)^{2k+i-j-1} \sum_{J \in \mathcal{C}_{i,j} \cap \E_+(P_{2k})}\EE\left[\ang{f,z\gamma}_{P_{i,j}}^J\right]\\
&\leq 2k\delta(pz)^{2k+i-j-1} \sum_{J \in \mathcal{C}_{i,j} \cap \E_+(P_{2k})}\EE\left[\ang{\gamma,z\gamma}_{P_{i,j}}^J\right]\\
&\leq 2k\delta(pz)^{2k+i-j-1} \cdot \big|\mathcal{C}_{i,j} \cap \E_+(P_{2k})\big|\cdot t_{P_{i,j}}(\gamma)\\
&\leq 2^{3k}p^{2k}\delta,
\end{align*}
where the last inequality used $t_{P_{i,j}}(\gamma)\leq(1+2k\delta)p^{j-i+1}$.
As $\{\C_{i,j}\cap \E_+(P_{2k}):1\leq i<j\leq 2k\}$ is a partition of $\E_+(P_{2k})$,
\begin{align*}
\Bigg|\EE\Bigg[\sum_{J \in \mathcal{E}_+ (P_{2k})} \ang{f,z\gamma}_{P_{2k}}^J\Bigg]
-\sum_{i=1}^{2k}(2k+1-i)T_i\Bigg|
\leq \sum_{1\leq i<j\leq 2k}\left|S_{i,j}-T_{j-i+1}\right|
\leq 2^{5k}p^{2k}\delta.
\end{align*}
Our final goal is to prove that $\sum_{i=1}^{2k}(2k+1-i)T_i$ is nonnegative for $0<z\leq 1$,
which suffices to conclude the proof.
We split this sum with respect to the parity to express it with the sum of nonnegative terms.
Namely,
\begin{align*}
\sum_{i=1}^{2k}(2k+1-i)T_i &= \sum_{i=1}^{k} (2k+1-2i)T_{2i} + (2k+2-2i) T_{2i-1}\\
   & =  \sum_{i=1}^{k} (2k+2-2i) T_{2i-1} + (k+1-i) T_{2i} + \sum_{i=0}^{k-1} (k-i)T_{2i}\\
   & = \sum_{i=1}^{k} 2(k+1-i) T_{2i-1} + (k+1-i) T_{2i} + \sum_{i=1}^{k} (k+1-i)T_{2i-2}\\
   & =  \sum_{i = 1}^k (k+1-i)(T_{2i} + 2 T_{2i-1} + T_{2i-2}).
\end{align*}
We claim  $T_{2\ell} + 2 T_{2\ell-1} + T_{2\ell-2}\geq 0$ for each $\ell=1,2,\dots,k$.
For that we may assume that $k=\ell$, as reducing $k$ by 1 only multiplies a factor of~$1/(pz)^2$ to each $T_{i}$ and hence, we may shorten the path $P_{2k}$ as long as it contains $P_{i,j}$. 
Let $L:=\{\{j,j+1\}:j=1,2,\dots,i\}$ and $R:=E(P_{2\ell})\setminus L$ for brevity,
that is, $L$ and $R$ are the left and the right half of the $2\ell$-edge path, respectively.
Let
\begin{align*}
\FF_{0}&:=\{J\in \C_{1,2\ell}:|J\cap L|\text{ and }|J\cap R|\text{ are even}\},\\
\G_{0}&:=\{J\in \C_{1,2\ell-1}:|J\cap L|\text{ and }|J\cap R|\text{ are even}\},\\
\text{ and }\quad\HH_{0}&:=\{J\in \C_{2,2\ell-1}:|J\cap L|\text{ and }|J\cap R|\text{ are even}\}. \end{align*}
Set $\FF_{0}^L:=\{J\cap L:J\in\FF_0\}$ and $\FF_{0}^R:=\{J\cap R:J\in\FF_0\}$.
Then $J\in \FF_0$ if and only if there exists $J_1\in \FF_0^L$ and $J_2\in \FF_0^R$ such that $J_1\cup J_2=J$. Let $h:=z\gamma$ and note
\begin{align*}
\sum_{J\in\FF_0}\ang{f,h}_{P_{2\ell}}^J
=\bigg(\sum_{J_1\in\FF_0^L}\ang{f,h}_{L}^{J_1}\bigg)\bigg(\sum_{J_2\in\FF_0^R}\ang{f,h}_{R}^{J_2}\bigg).
\end{align*}
For similarly defined $\G_0^L,\G_0^R,\HH_0^L$, and $\HH_0^R$ we  have the analogous identities.
Therefore,
\begin{align*}
\EE\bigg[\sum_{J\in\G_0}\ang{f,h}_{P_{2\ell}}^{J}\bigg]
&=\EE\bigg[\bigg(\sum_{J_1\in\G_0^L}\ang{f,h}_{L}^{J_1}\bigg)\bigg(\sum_{J_2\in\G_0^R}\ang{f,h}_{R}^{J_2}\bigg)\bigg]\\
&=\EE_x\bigg[\EE\bigg[\bigg(\sum_{J_1\in\G_0^L}\ang{f,h}_{L}^{J_1}\bigg)\bigg(\sum_{J_2\in\G_0^R}\ang{f,h}_{R}^{J_2}\bigg) \biggm\vert x_{\ell+1}=x\bigg]\bigg]\\
&=\EE_x\bigg[\EE\bigg[\sum_{J_1\in\G_0^L}\ang{f,h}_{L}^{J_1}\biggm\vert x_{\ell+1}=x\bigg]\EE\bigg[\sum_{J_2\in\G_0^R}\ang{f,h}_{R}^{J_2} \biggm\vert x_{\ell+1}=x\bigg]\bigg],
\end{align*}
where the last equality follows from the conditional independence of variables $\ang{f,h}_L^{J_1}$ 
and~$\ang{f,h}_R^{J_2}$ given $x_{\ell+1}=x$.
The key observation is that
\[
\FF_0^L=\G_0^L\qquad\text{ and }\qquad\G_0^R=\HH_0^R.
\]
Let $\phi(x):=\EE\big[\sum_{J_1\in\FF_0^L}\ang{f,h}_{L}^{J_1}\bigm\vert x_{\ell+1}=x\big]$
and $\psi(x):=\EE\big[\sum_{J_2\in\HH_0^R}\ang{f,h}_{R}^{J_2}\bigm\vert x_{\ell+1}=x\big]$
for brevity.
Then by the AM--GM inequality,
\begin{align}\label{eq:AM-GM}
\bigg|\EE\bigg[\sum_{J\in\G_0}\ang{f,h}_{P_{2\ell}}^{J}\bigg]\bigg|
\leq\EE_x\left[\big|\phi(x)\psi(x)\big|\right]
\leq \frac{1}{2}\left(\EE_x\left[\phi(x)^2\right]
+
\EE_x\left[\psi(x)^2\right]\right).
\end{align}
By the symmetry that maps the vertex $i$ to $2\ell+2-i$, we obtain
\[
\phi(x)=
\EE\bigg[\sum_{J_2\in\FF_0^R}\ang{f,h}_{R}^{J_2}\biggm\vert x_{\ell+1}=x\bigg]
\quad\text{ and }\quad
\psi(x)
=
\EE\bigg[\sum_{J_1\in\HH_0^L}\ang{f,h}_{L}^{J_1}\biggm\vert x_{\ell+1}=x\bigg],
\]
which implies $\EE\big[\sum_{J\in\FF_0}\ang{f,h}_{P_{2\ell}}^{J}\big]=\EE[\phi(x)^2]$
and $\EE\big[\sum_{J\in\HH_0}\ang{f,h}_{P_{2\ell}}^{J}\big]=\EE[\psi(x)^2]$, and thus,
\[
2\cdot\bigg|\EE\bigg[\sum_{J\in\G_0}\ang{f,h}_{P_{2\ell}}^{J}\bigg]\bigg|
\overset{\eqref{eq:AM-GM}}{\leq} \EE\bigg[\sum_{J\in\FF_0}\ang{f,h}_{P_{2\ell}}^{J}\bigg]
+\EE\bigg[\sum_{J\in\HH_0}\ang{f,h}_{P_{2\ell}}^{J}\bigg].
\]
We may do the same with $\FF_1,\G_1$, and $\HH_1$ defined by the odd intersections with two halves~$L$ and~$R$ to obtain
\[
2\cdot\bigg|\EE\bigg[\sum_{J\in\G_1}\ang{f,h}_{P_{2\ell}}^{J}\bigg]\bigg|
\leq \EE\bigg[\sum_{J\in\FF_1}\ang{f,h}_{P_{2\ell}}^{J}\bigg]
+\EE\bigg[\sum_{J\in\HH_1}\ang{f,h}_{P_{2\ell}}^{J}\bigg].
\]
Since $\FF_0\cup\FF_1$, $\G_0\cup\G_1$, and $\HH_0\cup\HH_1$ are partitions of $\C_{1,2\ell}$, $\C_{1,2\ell-1}$, and $\C_{2,2\ell-1}$, respectively,
we conclude that $T_{2\ell}+2T_{2\ell+1}+T_{2\ell-2}$ is nonnegative,
as claimed.
\end{proof}

Finally, we are ready to prove Lemma~\ref{lem:homcounting}.
\begin{proof}[Proof of Lemma~\ref{lem:homcounting}]
By Lemma~\ref{lem:diff} and~\ref{lem:better_sidorenko_lemma},
\begin{align*}
Q_{C_{2k+1}}(1;f)
&\geq (2k+1)\left(\int_{0}^1 p\cdot Q_{P_{2k}}(z;f) dz
-p^{2k+1}\left(\frac{\lambda}{\mu^{2k+1}d}+\frac{\lambda^{2k-1}n}{\mu^{2k+1}d^{2k}}\right)\right)\\
&\geq -(2k+1)p^{2k+1}\left(2^{5k}\delta
+\frac{\lambda}{\mu^{2k+1}d}+\frac{\lambda^{2k-1}n}{\mu^{2k+1}d^{2k}}\right)
\end{align*}
Thus, \eqref{eq:cancel} with $H=C_{2k+1}$ and \eqref{eq:homdensity} give
\begin{align*}
    t_{C_{2k+1}}(g)+t_{C_{2k+1}}(\gamma-g)
    &\geq \frac{1}{2^{2k}}  \left(p^{2k+1}\left(1- \frac{\lambda  }{\mu^{2k+1}d}-\frac{\lambda^{2k-1}n}{\mu^{2k+1}d^{2k}} -2k\delta\right)+
    Q_{C_{2k+1}}(1;f)\right)\\
    &\geq
     \frac{1}{2^{2k}}  p^{2k+1}\left(1- 2^{7k}\left(\frac{\lambda  }{\mu^{2k+1}d}+\frac{\lambda^{2k-1}n}{\mu^{2k+1}d^{2k}} +\delta\right)\right),
\end{align*}
as desired.
\end{proof}

\section{The relative Erd{\H{o}}s--Stone theorem for odd cycles}
To deduce Erd\H{o}s--Stone theorem from commonality, it is crucial to `regularise' the degree of the given subgraph $G$ of $\Gamma$ by restricting it to a vertex subset $X\subseteq V(\Gamma)$.
To this end, we employ an analogous argument to the proofs 
from~\cite{AHS14}*{Lemmas~4 and~6}.

\begin{lemma}\label{lem:regular}
For each $\varrho,\alpha> 0$ and $0 < \varepsilon < \alpha$, there exist $\eta > 0$ such that the following holds:
let $\Gamma$ be an $(n,d,\lambda)$-graph, with $\lambda \le \eta p^{1+\varrho} n$ and let $G \subset \Gamma$ be a subgraph satisying $e(G) \ge \alpha e(\Gamma)$. Then, there exists a set $X\subseteq V(G)$ such that
\begin{enumerate}[label=\rmlabel]
	\item\label{it:1}
	$|X| \ge \sqrt{ \varepsilon}n/8$,
	\item\label{it:2}
	$\deg_{G[X]}(x) \ge (\alpha - \varepsilon) p |X|$, and
	\item\label{it:3}
	$\deg_{\Gamma[X]}(x) = (1 \pm \varepsilon) p |X|$.
\end{enumerate}
\end{lemma}

We deduce Theorem~\ref{thm:main_turan} from Theorem~\ref{thm:common_injective_count}
and Lemma~\ref{lem:regular}.
Here we use subscripts such as $\tau_{a.b}$ to indicate that $\tau$ is the
parameter coming from Theorem $a.b$ or Lemma $a.b$.

\begin{proof}[Proof of Theorem~\ref{thm:main_turan}]
Suppose for a contradiction that $G$ contains no copy of $C_{2k+1}$.
Set 
$$
\varrho_{\ref{lem:regular}}=\frac{1}{2k-1},~~\alpha_{\ref{lem:regular}}=\frac{1}{2}+\delta,
~~\varepsilon_{\ref{lem:regular}}=\delta_{\ref{thm:common_injective_count}}=\frac{\delta^{2k}}{2^{11k}},
~~\text{ and }~~\mu_{\ref{thm:common_injective_count}}=\frac{\sqrt{\varepsilon_{\ref{lem:regular}}}}{8}$$ 
and let 
\[\eta=\min\{\eta_{\ref{thm:common_injective_count}},\eta_{\ref{lem:regular}},2^{-30k^2}\delta^{20k^3}\}.
\] 
Let $X\subset V(\Gamma)$ be the $\varepsilon_{\ref{lem:regular}}$-regular subset guaranteed by Lemma~\ref{lem:regular}.
Let $\overline{G}:=\Gamma\setminus G$ for brevity.
Theorem~\ref{thm:common_injective_count} combined with the $C_{2k+1}$-freeness of $G$
yields
\begin{align}\label{eq:lower}
    N_{C_{2k+1}}(\overline{G}[X]) = N_{C_{2k+1}}(G[X])+N_{C_{2k+1}}(\overline{G}[X])
 \geq \frac{1}{2^{2k}}p^{2k+1}\left(1-\frac{\delta^{2k}}{2^k}\right)|X|^{2k+1}.
\end{align}
The Expander Mixing Lemma implies
\begin{align*}
   \left| h_{C_{2k+1}}(\overline{G}[X]) - p\cdot (h_{P_{2k}}(\overline{G}[X]))^2\right|\leq \lambda\cdot h_{C_{2k}}(\Gamma).
\end{align*}
Since $\deg_{\overline{G}[X]}(x)\leq \tfrac{1}{2}(1-\delta)p|X|$ for each $x\in X$,
we have the bound $$h_{P_{2k}}(\overline{G}[X])\leq\frac{1}{2^{2k}}p^{2k}(1-\delta^{2k})|X|^{2k+1}$$ and thus, again using Lemma~\ref{lem: spectral argument},
\begin{align*}
   N_{C_{2k+1}}(\overline{G}[X]) &\leq  \frac{1}{2^{2k}}p^{2k+1}(1-\delta^{2k})|X|^{2k+1} + \lambda d^{2k}+\lambda^{2k-1}dn\\
   &\le\frac{1}{2^{2k}}p^{2k+1}|X|^{2k+1}\left(1-\delta^{2k}+\frac{2^{12k}\eta^{1/(2k-1)}}{\delta^{8k^2}}\right) .
\end{align*}
By the choice of  $\eta<2^{-30k^2}\delta^{20k^3}$, this contradicts to the lower bound~\eqref{eq:lower}.
\end{proof}

The proof of Lemma~\ref{lem:regular} is based on following lemma that appeared 
in~\cite{AHS14}*{Lemma~6}, where it was stated more generally for bijumbled graphs.

\begin{lemma}[{\cite{AHS14}*{p.\,8~(19)}}]\label{lem: part of lemma 6}
For all $\varrho,\alpha > 0$ and $0 < \varepsilon_1 < \alpha$ there exists an $\eta > 0$ such that the following holds:
let $\Gamma$ be an $(n,d,\lambda)$-graph, with $\lambda \le \eta p^{1+\varrho} n$ and let $G \subset \Gamma$ be a subgraph satisying $e(G) \ge \alpha  e(\Gamma)$. Then, there exists a set $Y\subset V(G)$ such that
\begin{enumerate}[label=\rmlabel]
	\item
	$|Y|\ge \sqrt{\frac{\varepsilon_1}{2}} n$ and
	\item
	$\deg_{G[Y]}(x) \ge (\alpha - \varepsilon_1) p |Y|$ for all $x \in Y$.\qed
\end{enumerate}
\end{lemma}

\begin{proof}[Proof of Lemma~\ref{lem:regular}]
Given  $\varrho> 0$ and $0 < \varepsilon < \alpha$, 
let $\varepsilon_1=\varepsilon/4$ and let $\eta_0$ be the $\eta$ obtained by Lemma~\ref{lem: part of lemma 6} applied with $\varrho$, $\alpha$, and $\varepsilon_1$.
Suppose $\eta\leq\eta_0$. We shall make $\eta$ smaller if necessary in what follows.
By Lemma~\ref{lem: part of lemma 6}, there exists a set $Y\subseteq V(G)$ such that $|Y|\ge \sqrt{\frac{\varepsilon_1}{2}} n$ and $\deg_{G[Y]}(x) \ge (\alpha - \varepsilon_1) p |Y|$ for all $x \in Y$.

We will iteratively delete vertices with deviating degrees.
Let $Y_-\subseteq Y$ be the subset of vertices with degree in $Y$ smaller than $(1-\varepsilon_1)p|Y|$.
Indeed, $e(Y,Y_-)<(1-\varepsilon_1)p|Y||Y_-|$.
However, the Expander Mixing Lemma guarantees
\begin{align*}
	 e(Y_-,Y)
	\ge p |Y||Y_-| - \lambda \sqrt{|Y||Y_-|}.
\end{align*}
Thus,
\begin{align*}
	\sqrt{|Y_-|}
	< \frac{\lambda \sqrt{|Y|}} { p |Y|- (1-\varepsilon_1) p |Y|}
	=  \frac{ \beta}{ \varepsilon_1 p  \sqrt{|Y|} }
	\le \frac{ \gamma p^ \varrho n}{\varepsilon_1   \sqrt{|Y|} },
\end{align*}
where the last inequality follows from $\lambda \le \eta p ^{1 + \varrho} n$.
Hence, by $|Y| \ge \sqrt{ \frac{\varepsilon_1}{2}} n$
\begin{align*}
	|Y_-|
<    \left( \frac{ \eta p^ \varrho \sqrt{|Y|} }{\varepsilon_1   \sqrt{ \frac{\varepsilon_1}{2}}}  \right)^2
	\le 2\eta^2 p^{2\varrho} |Y|/\varepsilon_1^{3}.
\end{align*}
Let $Y_+ \subseteq Y$ be the set of vertices with degree in $Y$ larger than $(1+\varepsilon_1)p|Y|$.
Repeating an analogous argument for $Y_+$ gives
\begin{align*}
	(1+\varepsilon_1)p|Y||Y_+|
	< e(Y_+,Y)
	\le p |Y||Y_+| + \lambda \sqrt{|Y||Y_+|},
\end{align*}
and thus, we have the same upper bound
\begin{align*}
	|Y_+|
	<   2\eta^2 p^{2\varrho} |Y|/\varepsilon_1^{3}.
\end{align*}

Now we are ready to start the deletion process. 
Let $Y_0 := Y_- \cup Y_+$ and for $i>0$ let 
\begin{align*}
	Y_{i+1} :=  \left\lbrace y \in Y \setminus \bigcup_{j=0}^{i} Y_j\colon e_{G[Y]}(y,Y_{i}) \ge \frac{1}{ 2^{i+2}}  \varepsilon_1 p |Y|  \right\rbrace,
\end{align*}
that is, $Y_{i+1}$ is the set of vertices that has `large' degree to the previously deleted vertices.
We claim that $|Y_i|$ decreases rapidly, as long as $\eta$ is small enough.
\begin{claim}\label{claim:iteration}
For each $i\ge 0$, if $\eta\leq \varepsilon_1^3/2^{2i+7}$, then
\begin{align}\label{iq:the new lemma 4 induction}
|Y_i| \le \eta^{i+1} p^{2(i+1)\varrho} |Y|.
\end{align}
\end{claim}
\begin{proof}[Proof of Claim~\ref{claim:iteration}]
The proof is by induction. As $\eta\leq 2\varepsilon_1^3$,
\begin{align}\label{iq:bound the high degree guys in Y}
	|Y_0|
	\le |Y_-| + |Y_+|
	< 4\eta^2 p^{2\varrho} |Y|/\varepsilon_1^3 \le \eta p^{2\varrho} |Y|
\end{align}
and hence,~\eqref{iq:the new lemma 4 induction} holds for $i=0$.
By definition, $e_\Gamma(Y_{i+1},Y_i) \ge \frac{\varepsilon_1}{2^{i+2}}  p |Y_i|  |Y_{i+1}|$.
By the Expander Mixing Lemma and $\lambda \le\eta  p^{1 + \varrho} n $,
\begin{align*}
	\frac{\varepsilon_1}{2^{i+2}}  p |Y_i|  |Y_{i+1}|
	\le  p |Y_i|  |Y_{i+1}| + \lambda \sqrt{ |Y_i|  |Y_{i+1}|}
	\le  p |Y_i|  |Y_{i+1}| + \eta p^{1+\varrho}n \sqrt{ |Y_i|  |Y_{i+1}|}.
\end{align*}
Thus,
\begin{align}\label{eq:iteration_shrink}
	|Y_{i+1}|
\le \frac{\eta n^2}{\left( \frac{\varepsilon_1}{ 2^{i+2}}  |Y|-|Y_i|\right)^2}\cdot \eta p^{2\varrho} |Y_i|.
\end{align}
By the induction hypothesis,
\begin{align*}
    \frac{\varepsilon_1}{ 2^{i+2}}|Y|-|Y_i|
    \geq \left(\frac{\varepsilon_1}{ 2^{i+2}}-\eta^{i+1}p^{2(i+1)\varrho} \right)\sqrt{\frac{\varepsilon_1}{2}} n
    \geq \left(\frac{\varepsilon_1}{ 2^{i+2}}-\eta\right)\sqrt{\frac{\varepsilon_1}{2}} n
    \geq \frac{\varepsilon_1^{3/2}}{2^{i+7/2}} n.
\end{align*}
Since $\eta\leq \varepsilon_1^{3}/2^{2i+7}$, we have $\eta n^2 \leq \left( \frac{\varepsilon_1}{ 2^{i+2}}  |Y|-|Y_i|\right)^2$.
The induction hypothesis and \eqref{eq:iteration_shrink} now prove the claim.
\end{proof}
As we cannot make $\eta$ smaller than $ 1/2^{2i+7}$ for arbitrary $i>0$, it is crucial to guarantee that the deletion process ends within a finite number of iterations. This is indeed true, and it allows us to take a judicious choice for $\eta$.

Let $K:= \frac{1}{2\varrho}-2$. Then the iteration terminates if $i\geq K$ and $\eta\leq \varepsilon_1/2$, since
\begin{align*}
    |Y_{i+1}|\leq \eta^{i+2}p^{(2i+4)\varrho}|Y|\leq 
    \frac{1}{2^{i+2}}\varepsilon_1 p|Y|,
\end{align*}
where the first inequality is by Claim~\ref{claim:iteration}.
Thus, we may take $\eta\leq \varepsilon_1^3/2^{3+1/\rho}$ so that Claim~\ref{claim:iteration} holds until the iteration terminates.

It remains to check that $X := Y \setminus \bigcup_{i=1}^{i^\ast} Y_i$ satisfies the three conditions of Lemma~\ref{lem:regular}.
Firstly,
\begin{align*}
	|X|
	= |Y|- \sum_{i=0}^{K} |Y_i|
	\ge \left( 1- \sum_{i=0}^{K} \eta^{i+1} p^{2(i+1)\varrho} \right) |Y|
	\ge  \left( 1- K  \eta\right)|Y|,
\end{align*}
so taking $\eta\leq 1/2K$ proves~\ref{it:1}, as $|Y|\geq  \sqrt{\frac{\varepsilon_1}{2}}n$.
Secondly, for $x \in X$,
\begin{align*}
	\deg_{\Gamma[X]}(x)
	< (1+ \varepsilon_1) p|Y|
	\le \frac{1+ \varepsilon_1 }{1- K \eta   }p|X|,
\end{align*}
and hence, letting $\eta \le \frac{\varepsilon_1}{K(1+2 \varepsilon_1)}$ proves the maximum degree condition in~\ref{it:3}.

For the proof of the minimum degree conditions in~\ref{it:2} and~\ref{it:3},
we estimate the number of deleted edges that are incident to each $x\in X$
and obtain
\begin{align*}
	\sum_{i=0}^{K} e_{\Gamma[Y]}(x,Y_i)\leq  |Y_{K}|+ \frac{\varepsilon_1}{2} p|Y| \sum_{i=0}^{\infty} \frac{1}{2^{i+1}}
	\le  \frac{1}{4}\varepsilon p|Y|\leq \frac{1}{2}\varepsilon p|X|.
\end{align*}
Therefore, as $x\notin Y_-$ by definition,
\begin{align*}
    \deg_{\Gamma[X]}(x)\geq \deg_{\Gamma[Y]}(x) - \sum_{i=0}^{K} e_{G[Y]}(x,Y_i)\geq (1-\varepsilon)p|X|.
\end{align*}
Similarly,
\begin{align*}
	\deg_{G[X]}(x)
	= \deg_{G[Y]}(x) - \sum_{i=0}^{K} e_{\Gamma[Y]}(x,Y_i)
	\ge (\alpha - \varepsilon )p|X|,
\end{align*}
which concludes the proof of the lemma.
\end{proof}

\end{document}